\documentclass[10pt]{amsart}
\usepackage[latin1]{inputenc}  
\usepackage[T1]{fontenc}

\usepackage{amscd, amsmath, amsthm, amssymb, amsfonts}
\usepackage{graphicx}
\usepackage{mathrsfs}
\usepackage[all]{xy}
\usepackage{verbatim}
\usepackage[pagebackref]{hyperref}
\usepackage{xcolor}
\usepackage{braket}
\usepackage[mathscr]{eucal}

\usepackage{mathtools}

\makeatletter\usepackage{microtype}\g@addto@macro\@verbatim{\microtypesetup{activate=false}}\makeatother%

\theoremstyle{plain}
\newtheorem{theorem}{Theorem}[section]
\newtheorem{lemma}[theorem]{Lemma}
\newtheorem{proposition}[theorem]{Proposition}
\newtheorem{conjecture}[theorem]{Conjecture} 
\newtheorem{corollary}[theorem]{Corollary} 
\newtheorem{claim}[theorem]{Claim}

\theoremstyle{definition}
\newtheorem{definition}[theorem]{Definition} 
\newtheorem{example}[theorem]{Example}

\theoremstyle{remark}
\newtheorem{remark}[theorem]{Remark}

\theoremstyle{plain}

\newtheorem*{thm*}{Th\'{e}or\`{e}me}

\newtheorem{lem}[theorem]{Lemma}
\newtheorem{pro}[theorem]{Proposition}
\newtheorem{cor}[theorem]{Corollary}

\newtheorem{ques}[theorem]{Question}

\theoremstyle{definition}

\numberwithin{equation}{section}

\newcommand\OO{{\mathcal{O}}}

\newcommand\CC{{\mathbb{C}}}
\newcommand\PP{{\mathbb{P}}}
\newcommand\QQ{{\mathbb{Q}}}
\newcommand\RR{{\mathbb{R}}}
\newcommand\ZZ{{\mathbb{Z}}}

\newcommand\Amp{{\rm Amp}} 
 \newcommand\Aut{{\rm Aut}}

\newcommand\Eff{{\rm Eff}} 

\newcommand\NE{\overline{{\rm NE}}}  
\newcommand\Nef{{\rm Nef}}
\newcommand\Nefe{{\rm Nef}^e} \newcommand\Nefp{{\rm Nef}^{+}} 

\newcommand\Mov{\overline{{\rm Mov}}}

\newcommand\Pic{\text{\rm Pic}}

\newcommand\GL{{\rm GL}} 
\newcommand\Ker{{\rm Ker}}
\newcommand\id{{\rm id}}

\newcommand{\ssec}{\subsection}

\newcommand{\sssec}{\subsubsection}

\newcommand{\ol}{\overline}
\newcommand{\ti}[1]{\widetilde{#1}}

\newcommand{\vast}{\bBigg@{4}}
\newcommand{\Vast}{\bBigg@{5}}

\newcommand{\cO}{\mathcal{O}}

\newcommand{\sX}{\mathscr{X}}
\newcommand{\sY}{\mathscr{Y}}

\newcommand{\cExt}{\mathcal{E}xt}
\newcommand{\cHom}{\mathcal{H}om}

\newcommand{\gD}{\Delta}

\newcommand{\gS}{\Sigma}

\newcommand{\gb}{\beta}

\newcommand{\gep}{\varepsilon}

\newcommand{\go}{\omega}
\newcommand{\gs}{\sigma}
\newcommand{\gt}{\theta}

\newcommand{\Hom}{\mathrm{Hom}}

\newcommand{\op}

\newcommand{\supp}{\mathrm{supp}}

\newcommand{\cnec}{\mathrel{:=}}
\newcommand{\cto}{\circlearrowleft}

\newcommand*\eto{%
	\xrightarrow[]{\raisebox{-0.25 em}{\smash{\ensuremath{\sim}}}}%
}
\newcommand{\hto}{\hookrightarrow}
\newcommand{\xto}[1]{\xrightarrow{ #1 }}

\begin{document}
	
	\title[Nef cone in Schoen's construction]{Nef cones of fiber products and an application to the Cone Conjecture}

	\author{C\'ecile Gachet}
	\address{Universit\'e C\^ote d'Azur, CNRS, LJAD, France; Institut f\"ur Mathematik, Humboldt-Universit\"at zu Berlin, Unter den Linden 6, 10099 Berlin, Germany}
    \email{cecile.gachet@hu-berlin.de}

	\author{Hsueh-Yung Lin}
	\address{Department of Mathematics, National Taiwan University, 
		and National Center for Theoretical Sciences,
		Taipei, Taiwan.}
	\email{hsuehyunglin@ntu.edu.tw}
	
	\author{Long Wang}
	\address{Shanghai Center for Mathematical Sciences, Fudan University, Jiangwan Campus, Shanghai, 200438, China; Graduate School of Mathematical Sciences, The University of Tokyo, 3-8-1 Komaba, Meguro-Ku, Tokyo 153-8914, Japan}
	\email{wanglll@fudan.edu.cn}

	\dedicatory{Dedicated to Professor Keiji Oguiso, on the occasion of his sixtieth birthday}

	\begin{abstract}
		
		We prove a decomposition theorem
		for the nef cone of smooth fiber products over curves, subject to the
		necessary condition that
		their N\'eron--Severi space decomposes.
		We apply it to describe the nef cone of
		so-called Schoen varieties, 
		which are the higher dimensional analogues 
		of the Calabi--Yau threefolds constructed by Schoen. 
		Schoen varieties give rise to Calabi--Yau pairs, 
		and in each dimension at least three, 
		there exist Schoen varieties
		with non-polyhedral nef cone.
		We prove the Kawamata--Morrison--Totaro
		Cone Conjecture for the nef cones of Schoen varieties, 
		which generalizes the work by 
		Grassi and Morrison.
	\end{abstract}
	
	

	\maketitle

	\thispagestyle{empty}

	\section{Introduction}
	
	\ssec{Cone Conjecture}
	To understand the geometry of a smooth projective variety $X$, studying the Mori cone of curves $\NE(X)$ and its dual, the nef cone $\Nef(X)$, is central, especially from the viewpoint of the minimal model program (MMP).
	
	An important part of the relationship between the Mori cone and the MMP is 
	captured by 
	the Cone Theorem, and the Contraction Theorem.
	These theorems assert that the $K_X$-negative part
	of the Mori cone of a smooth projective variety $X$ is rational polyhedral away from the $K_X$-trivial hyperplane, 
	and the extremal rays of the $K_X$-negative part
	correspond to some morphisms from $X$, involved in the MMP.
	In particular,	when $X$ is a Fano variety (namely, $-K_X$ is ample), 
	the cone $\Nef(X)$ is a rational polyhedral cone, 
	and its extremal rays are generated by semiample classes. 
	In general, however, it is difficult to describe the whole Mori cone, 
	or dually the whole nef cone, even under the slightly weaker 
	assumption that $-K_X$ is semiample. 
	For instance, if $X$ is the blowup of $\PP^2$ at the base points of 
	a general pencil of cubic curves in $\PP^2$, then $-K_X$ is semiample but $\Nef(X)$ is not rational polyhedral.

	When $X$ is $K$-trivial,
	we expect nevertheless that some essential parts of the nef cone of $X$ are rational polyhedral, up to the action of $\Aut(X)$. A precise statement, known as the Cone Conjecture,	was first formulated by Morrison \cite{Mo93} and Kawamata \cite{Ka97}. 
	It was later generalized by Totaro~\cite{To10}
	to klt Calabi--Yau pairs $(X,\gD)$
	(see Section~\ref{ssec-kltCY}), 
	thus including many more examples,
	already in dimension $2$. When stated by these authors, the Cone Conjecture comprises predictions both on the nef cone and on the movable cone of varieties, and has both an absolute and a relative version. In what follows, we will only consider the absolute Cone Conjecture for nef cones of certain Calabi--Yau pairs.
	
	Let us recall the statement
	formulated by Totaro in~\cite[Conjecture 2.1]{To10}, starting with some notations.
	For a pair $(X, \gD)$, we define 
	$$\Aut(X,\gD) \cnec \Set{ f \in \Aut(X) | f(\supp(\gD)) = \supp(\gD) }.$$
	We also define the \textit{nef effective cone} 
	$\Nefe(X)$ as
	\[ \Nefe(X) := \Nef(X)\cap \Eff(X), \]
	where $\Eff(X)$ is the effective cone of $X$.

	\begin{conjecture}[Kawamata--Morrison--Totaro Cone Conjecture]\label{cone} Let $(X, \Delta)$ be a klt Calabi--Yau pair. 
		There exists a rational polyhedral cone $\Pi$ in $\Nefe(X)$ which is a fundamental domain for the action of $\Aut(X, \Delta)$ on $\Nefe(X)$, in the sense that
		\[ \Nefe(X) =\bigcup_{g\in\Aut(X, \Delta)}g^{\ast}\Pi, \]
		and $\Pi^{\circ} \cap (g^{\ast}\Pi)^{\circ} = \varnothing$ unless $g^{\ast} = \mathrm{id}$.
	\end{conjecture}
	
	An important prediction of 
	the Cone Conjecture for the MMP is that
	the number of $\Aut(X, \Delta)$-equivalence classes of faces of the nef effective cone $\Nefe(X)$ corresponding to birational contractions or fiber space structures is finite (see e.g.~\cite[p.243]{To10}).

	Note that it is standard to replace Conjecture~\ref{cone} 
	by the {\it a priori} stronger following conjecture. Let $\Nefp(X)$ denote the convex hull of 
	\[ \Nef(X)\cap N^1(X)_{\QQ},\]
	where $N^1(X)_{\QQ}$ is the rational N\'eron--Severi space of $X$.

	\begin{conjecture}\label{conep}
		Let $(X, \Delta)$ be a klt Calabi--Yau pair. Then the following statements hold. 
		\begin{enumerate}
			\item  There exists a rational polyhedral cone in $\Nefp(X)$ which is a fundamental domain for the action of $\Aut(X, \Delta)$ on $\Nefp(X)$. 
			\item 
			We have
			$$\Nefp(X) = \Nefe(X).$$
		\end{enumerate}
	\end{conjecture}
	
	Thanks to the fundamental work of Looijenga~\cite{Lo14}, we prove that the two conjectures are equivalent (see Corollary~\ref{cor-equivconj}).
	
	\ssec{Nef cones of fiber products}

	The starting point of this work is a decomposition theorem for the nef cone of a fiber product over a curve.
	
	It begins with the following general question.
	Let $W_1$ and $W_2$ be projective varieties and let
	$\phi_1 : W_1 \to B$ and $\phi_2 : W_2 \to B$ be 
	surjective morphisms over a base $B$.
	Assume that
	the fiber product $W \cnec W_1 \times_B W_2$ is irreducible.
	
	\begin{ques}\label{que-decompNef}
		Denote by $p_i : W \to W_i$ the natural projections.
		When do we have
		\begin{equation}\label{decomp-Nef}
			p_1^* \Nef(W_1) + p_2^* \Nef(W_2) = \Nef(W)?
		\end{equation}
	\end{ques}

	As the nef cone of a projective variety linearly
	spans the whole space of numerical classes of $\RR$-divisors on the variety,
	the nef cone decomposition~\eqref{decomp-Nef} exists only if
	\begin{equation}\label{decomp-N1}
		p_1^*N^1(W_1)_{\RR} + p_2^*N^1(W_2)_{\RR} = N^1(W)_{\RR}.
	\end{equation}
	We may then ask which fiber products satisfying the decomposition~\eqref{decomp-N1} also have the decomposition~\eqref{decomp-Nef}.

	When $B$ is a point, it is not hard to see that~\eqref{decomp-N1} implies~\eqref{decomp-Nef}.
	In this case indeed, the decomposition~\eqref{decomp-N1} is a direct sum. Every divisor $D$ decomposes uniquely as $p_1^{\ast}D_1 + p_2^{\ast}D_2$, where $D_1=D_{|W_1\times\{\mathrm{pt}\}}$ and $D_2=D_{|\{\mathrm{pt}\}\times W_2}$. If $D$ is nef, then so are its restrictions, and hence~\eqref{decomp-Nef} follows.
	When $B$ is $\PP^1$ and the varieties $W_i$ are certain rational elliptic surfaces, the decomposition~\eqref{decomp-Nef} was proven in~\cite[Proposition 3.1]{GM93}.
	We show that the implication~\eqref{decomp-N1} $\Rightarrow$~\eqref{decomp-Nef}
	continues to hold for an arbitrary irreducible
	fiber product over a curve.
	
	\begin{theorem}\label{thm-nefdec} For $i=1,2$, let $\phi_i : W_i \to B$ be a surjective morphism from a projective variety to a projective curve $B$. Assume that
		\begin{enumerate}
			\item  The fiber product $W=W_1 \times_B W_2$ is irreducible.
			\item  We have $$p_1^*N^1(W_1)_{\RR}+p_2^*N^1(W_2)_{\RR}=N^1(W)_{\RR}.$$ 
		\end{enumerate}
		Then 
		$$p_1^*\Nef(W_1) + p_2^*\Nef(W_2)=\Nef(W).$$
		As a consequence, we also have
		$ p_1^*\Amp(W_1) + p_2^*\Amp(W_2)=\Amp(W)$.
	\end{theorem}

	In Examples~\ref{ex-nondecomp},~\ref{ex-nondecompflat}, and~\ref{ex-nondecompbis},
	we build examples of 
	fiber products over bases of higher dimension, that fail the implication ~\eqref{decomp-N1} $\Rightarrow$~\eqref{decomp-Nef}.  
	In Remark~\ref{rem-exMov}, we recall a classical example emphasizing that a similar decomposition does not 
	hold for the movable cone of divisors of a fiber product over a curve.
	
	We establish the following corollary to this first theorem.
	
	\begin{cor}\label{cor-extr}
		Keep the notations and assumptions of Theorem~\ref{thm-nefdec}.
		Then the extremal rays of the convex cone $\Nef(W)$ are exactly the pullbacks of the extremal rays of the two cones $\Nef(W_1)$ and $\Nef(W_2)$.
		In particular, the cone	$\Nef(W)$ is rational polyhedral
		if and only if the cones $\Nef(W_1)$
		and $\Nef(W_2)$ are both rational polyhedral.
	\end{cor}
	
	This corollary can be seen as a means to
	construct fiber products over curves, whose nef cones are not rational polyhedral.
	
	\ssec{Cone Conjecture for Schoen varieties}
	
	Among the strict Calabi--Yau manifolds (see Definition \ref{def-cymfd}) whose nef cones are 
	known to not be rational polyhedral, 
	to our knowledge, the Cone Conjecture is only known so far in two special cases.  
	One of them is the desingularized Horrocks--Mumford quintics, studied by Borcea in \cite{Bo91} (see also \cite{Fr01}); the other is the fiber product of two general rational elliptic surfaces with sections over $\PP^1$, constructed by Schoen in \cite{Sc88}, and investigated by Namikawa and Grassi--Morrison~\cite{Na91,GM93}. 
	Both examples are of dimension three.
	
	The main goal of this paper is to 
	prove the Cone Conjecture for generalizations
	of Schoen's Calabi--Yau threefolds, typically Calabi--Yau pairs, but also higher-dimensional strict Calabi--Yau varieties. In both cases, the underlying varieties, which we call {\it Schoen varieties}, are constructed as fiber products over $\PP^1$.
	
	Let us first summarize our construction defining Schoen varieties; 
	we refer to Subsections \ref{subsec-Wconstr} and \ref{subsec-Xconstr} for more details. 
	We start with Fano manifolds $Z_1$ and $Z_2$ of dimension at least two, which respectively admit an ample and globally generated divisor $D_i$ ($i=1,2$), such that $-(K_{Z_i}+D_i)$ is globally generated. 
	We take $W_i \subset \PP^1\times Z_i$ to be a general member in the linear system $|\OO_{\PP^1}(1)\boxtimes\OO_{Z_i}(D_i)|$. There is a fibration $\phi_i: W_i \to \PP^1$. We put another mild condition on the fibrations $\phi_1$ and $\phi_2$ to be general with respect to one another (see the second paragraph of Subsection \ref{subsec-Xconstr} for a precise statement). Consider the following fiber product over $\PP^1$
	\[ \phi: X:= W_1 \times_{\PP^1} W_2 \to \PP^1. \]
	Under our assumptions, the variety $X$ is smooth and projective. All varieties obtained through this procedure are called {\it Schoen varieties}.
	
	It follows from the construction that $-K_X$ is globally generated, so many effective $\QQ$-divisors are $\QQ$-linearly equivalent to $-K_X$.
	Any such $\QQ$-divisor $\gD$ yields a Calabi--Yau pair $(X,\gD)$, that we call a {\it Schoen pair}.
	
	We prove the following result.
	
	\begin{theorem}\label{main} 
		Let $(X,\gD)$ be a Schoen pair.
		Then there exists a rational polyhedral fundamental domain for the action of $\Aut(X, \Delta)$ on $\Nefe(X) =\Nefp(X) = \Nef(X)$.
	\end{theorem}
	
	Note that, by Corollary \ref{cor-extr},
	the cone $\Nef(X)$ is not rational polyhedral
	as soon as one of the cones $\Nef(W_i)$ ($i=1,2$) is not, typically if one of the factors $W_i$ is a rational elliptic surface with $Z_i\simeq\PP^2$ and $D_i=\OO_{\PP^2}(3)$. Using this remark, we provide in Example \ref{ex-infinite} the first series of strict Calabi--Yau manifolds (and Calabi-Yau pairs) of arbitrary dimension for which the Cone Conjecture holds, with nef cones that are not rational polyhedral.

	We finally mention two unsurprising consequences of Theorem \ref{main} (see Corollary~\ref{cor_finite}): The finite presentation of the group of components $\pi_0\Aut(X)$, and the finiteness of real forms on $X$, up to isomorphism.

	\subsection{Relation to other work}

	\sssec{Cone Conjecture}
	
	We refer to~\cite{LOP18} and the references therein
	for a survey of the Cone Conjecture for varieties (as opposed to pairs).
	As for the Cone Conjecture for Calabi--Yau pairs,
	its $2$-dimensional case was proven by Totaro~\cite{To10}.
	Kopper \cite{Ko20} also proved the Cone Conjecture for Calabi--Yau pairs arising from Hilbert schemes of points on certain rational elliptic surfaces; the underlying varieties in his work may have non rational polyhedral nef cones, but they only appear in even dimensions.
	The references~\cite{FHS21, LZ22} also contain some recent results.

	\sssec{Cone Conjectures for varieties with rational polyhedral nef cones}
	
	One way of proving the Cone Conjecture
	for a smooth projective variety $X$ is to show that $\Nef(X)$ is a rational polyhedral cone
	and that $\Nef(X) = \Nefe(X)$ (see e.g.~\cite[Proposition 6.5]{La13}).
	This is the case whenever $X$ is a smooth anticanonical hypersurface in a Fano manifold $Y$
	of dimension at least $4$, by the following theorem, due to~
	Koll\'{a}r \cite[Appendix]{Ko91}.
	
	\begin{theorem}\label{kol} Let $D$ be a smooth anticanonical hypersurface in a smooth Fano variety $Y$ of dimension at least $4$. Then the natural restriction map $\Nef(Y)\to\Nef(D)$ is an isomorphism. In particular, $\Nef(D)$ is a rational polyhedral cone, generated by classes of semiample divisors.
	\end{theorem}
	
	Other Calabi--Yau pairs $(X, \Delta)$ for which $\Nef(X) = \Nefe(X)$ is rational polyhedral are described in the work of Coskun and Prendergast-Smith~\cite{PS12b, CPS14, CPS19}.
	
	\sssec{Fiber product constructions}
	
	Constructing Calabi--Yau threefolds as
	fiber products of two general rational elliptic surfaces with sections over $\PP^1$ was 
	first considered and investigated by Schoen~\cite{Sc88}.
	It recently came back to light as Suzuki considered a certain higher-dimensional generalization of Schoen's construction and studied its arithmetic properties in \cite{Su21}. Similar ideas are also involved in  
	Sano's constructions of non-K\"{a}hler Calabi--Yau manifolds with arbitrarily large second Betti number in \cite{Sa21}.

	\sssec{Cone conjecture for movable cones}
	We have already mentioned that there is a part of the Cone Conjecture concerned with movable cones~\cite[Conjecture 2.1.(2)]{To10}. It predicts 
	that a Calabi--Yau variety should have
	finitely many minimal models, up to isomorphism~\cite[Theorem 2.14]{CL14}). See~\cite{Mo96, Ka97, To10, LOP18} for related references. This part of the Cone Conjecture was verified for some cases. Notably, in \cite{CO15}, Cantat and Oguiso produced the first series of strict Calabi--Yau manifolds in arbitrary dimension whose movable cones are not rational polyhedral and for which the Cone Conjecture for movable cones holds. We refer to \cite{HSY22, ILW22, LW22, Wa22} and references therein for more results.

	In \cite{Na91} Namikawa showed that a certain strict Calabi--Yau threefold, constructed as a Schoen variety, has finitely many minimal models, up to isomorphism. Nonetheless, the Cone Conjecture is still unknown for the movable cone of divisors of this Calabi--Yau threefold. The Cone Conjecture for the nef cones of each of these minimal models is not known either. Similar questions could be asked for the Schoen varieties of higher dimension constructed here.

	\subsection{Structure of the paper}
	Section~\ref{prel} is devoted to some preliminaries and fundamental results. 
	We prove Theorem \ref{thm-nefdec} in Section \ref{dec}. 
	After constructing Schoen varieties and Schoen pairs in Section~\ref{cons},
	we prove Theorem \ref{main} in Section~\ref{prof}.

	\subsection*{Acknowledgments} 
	We thank Professors Serge Cantat, Tien-Cuong Dinh, Ching-Jui Lai, Vladimir Lazi\'c, Keiji Oguiso, Burt Totaro, Hokuto Uehara, and Claire Voisin for their questions, comments, suggestions, and encouragement. We are also grateful to the referee for careful reading and useful comments.
	The first author would like to thank JSPS Summer Program for providing the opportunity to visit the third author in Tokyo, where this paper was written. The third author would like to thank Department of Mathematics at National University of Singapore, Professor De-Qi Zhang and Doctor Jia Jia for warm hospitality.
	The first author is supported by the ERC Advanced Grant SYZYGY. This project has received funding from the European Research Council (ERC) under the EU Horizon 2020 program (grant agreement No. 834172).
	The second author is supported by
	the Ministry of Education Yushan Young Scholar Fellowship (NTU-110VV006)
	and the National Science and Technology Council
	(110-2628-M-002-006-).
	The third author is supported by JSPS KAKENHI Grant (21J10242), Postdoctoral Fellowship Program of CPSF (GZC20230535), and National Key Research and Development Program of China (\#2023YFA1010600).

	\section{Preliminaries}\label{prel}

	We work over the field $\CC$ of complex numbers throughout this paper. For notions of birational geometry, we refer to ~\cite{KM98}. 
	
	\subsection{Notations}\label{ssec-notation} Let $X$ be a projective variety. 
	We write $N^1(X)$ for the free abelian group generated by the classes of Cartier divisors modulo numerical equivalence. 
	
	Inside the vector space $N^1(X)_{\RR}:=N^1(X)\otimes \RR$, we denote by $\Nef(X)$ the \textit{nef cone}, i.e., the closure of the ample cone $\Amp(X)$, and by $\Eff(X)$ the \textit{effective cone}. The \textit{nef effective cone} $\Nefe(X)$ is defined as
	\[ \Nefe(X) := \Nef(X)\cap \Eff(X). \]
	Let $\Nefp(X)$ denote the convex hull of 
	\[ \Nef(X)\cap N^1(X)_{\QQ}, \]
	where $N^1(X)_{\QQ}:=N^1(X)\otimes \QQ$. We denote by $N_1(X)$ the group of $1$-cycles modulo numerical equivalence. The intersection product defines a perfect pairing between the two vector spaces $N^1(X)_{\RR}$ and $N_1(X)_{\RR}$. Under this pairing, the nef cone $\Nef(X)$ is dual to the \textit{Mori cone} $\NE(X)$, 
	which is by definition the closure of the convex cone of effective $1$-cycles in $N_1(X)_{\RR}$.

	The group of automorphisms of $X$ is denoted by $\Aut(X)$, and acts on $N^1(X)$ by pullback. This action 
	\[ \rho: \Aut(X) \to \GL(N^1(X))
	\]
	linearly  extends to $N^1(X)_{\RR}$, preserving the cones $\Nefe(X)$ and $\Nefp(X)$. The connected component of the identity in $\Aut(X)$ is a normal subgroup $\Aut^0(X)$, which acts trivially on $N^1(X)$ \cite[Lemma 2.8]{Br18}.

	\subsection{Klt Calabi--Yau pairs}\label{ssec-kltCY}
	
	A \textit{pair} is the data $(X, \Delta)$ of a normal projective variety $X$ together with an effective $\RR$-divisor $\Delta$ on $X$ such that $K_X + \Delta$ is $\RR$-Cartier.
	
	\begin{definition}\label{def-cypair} Following \cite{To10}, we say that a pair $(X, \Delta)$ is \textit{Calabi--Yau} if $X$ is $\QQ$-factorial and $K_X + \Delta$ is numerically trivial.
	\end{definition}
	
	Let us briefly recall the definition of a Kawamata log terminal (klt) pair. We start with a notation.
	For any pair $(X,\Delta)$ and any birational morphism $\mu : \ti{X} \to X$, there exists a unique $\RR$-divisor $\ti{\gD}$ on $\ti{X}$ such that $$K_{\ti{X}} + \ti{\gD} = \mu^*(K_X + \gD) \ \text{ and }\ \mu_*\ti{\gD} = \gD.$$ 
	
	A pair $(X, \Delta)$ is called \textit{klt}
	if, for any birational morphism $\mu : \ti{X} \to X$, when defining the divisor $\ti{\gD}$ as above, each irreducible component of $\ti{\gD}$ has coefficient 
	less than one.
	
	Note that if we can find one resolution of singularities $\mu : \ti{X} \to X$ whose corresponding divisor $\ti{\gD}$ has simple normal crossings, with irreducible components of coefficients less than one, then $(X,\gD)$ is klt.
	
	\begin{definition}\label{def-cymfd}
		Let $X$ be a smooth projective variety. We say that $X$ is a {\it Calabi--Yau manifold} 
		if its canonical line bundle $K_X$ is trivial and $h^i(X,\OO_X)=0$ for any $0<i<\dim X$. If in addition, $X$ is simply connected, we call it a {\it strict Calabi--Yau manifold}.
	\end{definition}
	
	\subsection{Looijenga's result}

	The following result is crucial in this paper.

	\begin{proposition}\label{pro-looij}
		
		Let $X$ be a projective variety 
		and let $H \le \Aut(X)$ be a subgroup.
		Assume that there is a rational polyhedral cone $\Pi\subset \Nefp(X)$ such that 
		$\Amp(X)\subset H \cdot\Pi.$ Then
		\begin{enumerate}
			\item $H\cdot \Pi = \Nefp(X)$, and the $H$-action on $\Nefp(X)$ has a rational polyhedral fundamental domain.  
			\item The group $\rho(H)$ is finitely presented. 
		\end{enumerate}
	\end{proposition}
	
	This result should be well-known to experts, but we include a proof
	for the sake of completeness.
	It relies on the
	fundamental results due to Looijenga~\cite[Proposition 4.1, Application 4.14, and Corollary 4.15]{Lo14}, which we extract and formulate here as Lemma~\ref{looij}.
	Recall that a convex cone $C \subset N_{\RR}$ in a finite dimensional $\RR$-vector space $N_{\RR}$ is called {\it strict} if its closure $\ol{C} \subset N_{\RR}$
	contains no line.
	
	\begin{lemma}\label{looij} 
		
		Let $N$ be a finitely generated free $\ZZ$-module, and let $C$ be a strict convex open cone in the $\RR$-vector space $N_{\RR}:= N\otimes \RR$. 
		Let $C^{+}$ be the convex hull of $\overline{C} \cap N_{\QQ}$.
		Let $(C^\vee)^\circ \subset N_{\RR}^\vee$ be the interior of the dual cone of $C$. 
		Let $\Gamma$ be a subgroup of $\GL(N)$ which preserves the cone $C$. Suppose that
		\begin{itemize}
			\item there is a rational polyhedral cone $\Pi \subset C^+$ such that $C \subset \Gamma\cdot\Pi$;
			\item there exists an element $\xi \in (C^\vee)^\circ \cap N_{\QQ}^\vee$ whose stabilizer in 
			$\Gamma$ (with respect to the dual action
			$\Gamma \cto N_{\QQ}^\vee$) is trivial.
		\end{itemize}	
		Then $\Gamma \cdot \Pi = C^+$
		and the $\Gamma$-action on $C^{+}$ has 
		a rational polyhedral fundamental domain. 
		Moreover, the group $\Gamma$ is finitely presented.
	\end{lemma}
	
	To prove Proposition \ref{pro-looij}, it is key to connect abstract convex geometry as in Lemma \ref{looij} with the specifics of an automorphism group acting on an ample cone. That is the goal of the next lemma.

	\begin{lem}\label{lem-trivstab}
		Let $X$ be a projective variety.
		Then there exists an ample Cartier divisor on $X$, whose numerical class $\eta \in N^1(X)$ satisfies:
		For every $g \in \Aut(X)$,
		if $g^*\eta = \eta$, then
		$g^*$ is the identity on $N^1(X)$.
	\end{lem}
	
	\begin{proof}
		Our proof is inspired by the argument
		of \cite[Proposition 6.5]{La13}.
		
		Let $\Gamma \cnec \rho(\Aut(X)) < \GL(N^1(X))$.
		For every $\gt \in N^1(X)_{\QQ}$,
		let $\Gamma_{\gt}$ denote the subgroup of $\Gamma$ stabilizing $\gt$. We want to find an element $\eta \in \Amp(X)\cap N^1(X)$ such that
		$\Gamma_\eta$ is trivial. By linearity, it is sufficient to find such an element in $\Amp(X)\cap N^1(X)_{\QQ}$.
		
		By Fujiki--Liebermann's theorem~\cite[Theorem 2.10]{Br18}, for every element $\gt$ in $\Amp(X)\cap N^1(X)_{\QQ}$, the stabilizer $\Gamma_{\gt}$ is finite.
		Pick an element $\eta\in \Amp(X)\cap N^1(X)_{\QQ}$ 
		such that $\Gamma_{\eta}$ has the smallest possible order. Since the discrete set $N^1(X)$ is preserved by the action of $\Gamma$, there is an open neighborhood $U \subset \Amp(X)$ of $\eta$ such that, for every $\gamma\in\Gamma\setminus\Gamma_{\eta}$, the intersection $\gamma U \cap U$ is empty. In particular, for every $\gt \in U\cap N^1(X)_{\QQ}$, we have $\Gamma_{\gt} \subset \Gamma_{\eta}$, so $\Gamma_{\gt}  =  \Gamma_{\eta}$ by the minimality assumption on $\eta$. Hence, we have $$\gamma_{|U\cap N^1(X)_{\QQ}} = \mathrm{id}_{U\cap N^1(X)_{\QQ}},$$ which extends by linearity to $\gamma = \mathrm{id}$. So the stabilizer $\Gamma_{\eta}$ is trivial, which concludes the proof.
	\end{proof}
	
	We can now establish Proposition \ref{pro-looij}.
	
	\begin{proof}[Proof of Proposition \ref{pro-looij}] 
		
		Let us set $N=N^1(X)$, $C=\Amp(X)$, and $\Gamma=\rho(H)$. 
		To apply Lemma \ref{looij} in this set-up, it suffices to
		construct an element $\xi_0\in (C^{\vee})^{\circ} \cap N^{\vee}_{\QQ}$ with trivial stabilizer with respect to the dually induced $\Gamma$-action. Start by picking any $\xi\in (C^{\vee})^{\circ} \cap N^{\vee}_{\QQ}$.
		
		The idea is to find a minimizer $\eta$ for the linear functional $\xi$ on the set
		$$\gS \cnec \Set{\eta \in C \cap N | \Gamma_{\eta} \text{ is trivial}},$$
		and to relate the stabilizer of $\eta$ (which is then trivial by construction) to the stabilizer of $\xi$ (which we want to be trivial).
		
		Note that $\Sigma$ is non-empty by Lemma~\ref{lem-trivstab}, and discrete. By definition, the linear form $\xi$ takes positive values on the whole convex set $\ol{C}\setminus\{0\}$. Picking a large enough positive integer $r$, the intersection
		$$\Sigma\cap\{x\in \ol{C}\mid \xi(x)\le r\}$$
		is now non-empty and finite. Minimizing $\xi$ on this finite set is equivalent to minimizing it on $\Sigma$, and thus $\xi$ has finitely many minimizers in $\Sigma$.
		
		Since $C\cap N$ is discrete, we can now perturb $\xi$ into a new linear form $\xi_0\in (C^{\vee})^{\circ} \cap N^{\vee}_{\QQ}$, which has exactly one minimizer $\eta$ on $\Sigma$.    
		As the set $\gS$ is $\Gamma$-invariant and as $\Gamma_{\eta}$ is trivial, we have, for any non-trivial $\gamma\in\Gamma$, that $\gamma\eta\in\Sigma\setminus\{\eta\}$, and in particular
		$$(\gamma \xi_0)(\eta) = \xi_0(\gamma \eta) > \xi_0(\eta).$$
		So the stabilizer of $\xi_0$ in $\Gamma$ is trivial.
	\end{proof}
	
	We prove a simple corollary of Proposition \ref{pro-looij}.
	
	\begin{cor}\label{cor-equivconj} 
		Conjecture~\ref{cone} and Conjecture~\ref{conep} are equivalent.
	\end{cor}
	
	\begin{proof}
		Clearly, Conjecture~\ref{conep} implies
		Conjecture~\ref{cone}.
		Now, fix a pair $(X,\gD)$ for which Conjecture~\ref{cone} holds. Let $\Pi\subset \Nefe(X)$ be a rational polyhedral fundamental domain for the action of $\Aut(X,\gD)$ on $\Nefe(X)$. Then $\Pi\subset \Nefp(X)$ by definition of $\Nefp(X)$. By Proposition \ref{pro-looij}.(1), 
		$$\Nefe(X) = \Aut(X,\gD) \cdot \Pi = \Nefp(X).$$
		So Conjecture~\ref{conep} holds. 
	\end{proof}

	\section{The nef cone of a fiber product over a curve}\label{dec}
	
	In this section, we prove Theorem \ref{thm-nefdec}. Let us recall the notations.
	For $i=1,2$, the map $\phi_i : W_i \to B$ is a surjective morphism from a projective variety to a projective curve $B$. We consider the fiber product
	\[ \xymatrix@=1.5em{ & W = W_1 \times_{B} W_2 \ar[dl]^{p_1} \ar[dd]_{p} \ar[dr]_{p_2}   \\
		W_1 \ar[dr]_{\phi_1} & &     W_2 \ . \ar[dl]^{\phi_2}   \\
		& B } \]
	and work under the following assumptions:
	\begin{enumerate}
		\item  The fiber product $W = W_1 \times_{B} W_2$ is  irreducible;
		\item For every $D \in N^1(W)_{\RR}$, there exist $D_1 \in N^1(W_1)_{\RR}$ and $D_2 \in N^1(W_2)_{\RR}$ such that \[ D = p_1^*D_1 + p_2^*D_2. \]
	\end{enumerate}

	\begin{proof}[Proof of Theorem \ref{thm-nefdec}.]
		Let us fix $D \in \Nef(W)$ and consider a decomposition in real classes 
		$$D = p_1^*D_1 + p_2^*D_2 \in N^1(W)_{\RR}$$
		as in Assumption (2) right above. We prove three lemmas regarding the positivity of these two summands $D_1$ and $D_2$.
		
		\begin{lem}\label{lem-vert}
			Fix $i=1,2$. Let $C_i$ be a curve contained in a fiber of $\phi_i: W_i\to B$. Then $ D_i \cdot C_i \ge 0$.
		\end{lem}
		\begin{proof}
			By symmetry, we can focus on $i = 1$. Fix any point $s \in \phi_2^{-1}(\phi_1(C_1))$ and consider the fiber product	$\ti{C_1} \cnec C_1 \times_{B} \{s\}$, which can be seen as a curve in $W_1$.
			We have
			$$0 \le D \cdot \ti{C_1} = (p_1^*D_1 + p_2^*D_2) \cdot \ti{C_1}
			= D_1 \cdot p_{1*}\ti{C_1} + D_2 \cdot p_{2*}\ti{C_1} = D_1 \cdot C_1.$$
			This proves the lemma. 
		\end{proof}
		
		\begin{lem}\label{lem-nef}
			Either $D_1$ or $D_2$ is nef.
		\end{lem}
		\begin{proof}
			Assume by contradiction that both $D_1$ and $D_2$ are not nef.
			Then for each $i$, there exists a curve $C_i$ in $W_i$
			such that $D_i \cdot C_i < 0$.
			Note that since the fiber product $W$ is assumed to be irreducible, the base $B$ is also irreducible. Hence, and
			by Lemma~\ref{lem-vert}, 
			we have $\phi_i(C_i) = B$. So the (possibly reducible) fiber product $C_1 \times_{B} C_2$ contains a curve $\ti{C}$ dominating $B$.
			Let $\gb_1,\gb_2 \in \ZZ_{>0}$
			be such that $p_{i*}\ti{C} = \gb_i C_i$. 
			Then on one hand, 
			$$\gb_1 D_1 \cdot C_1  + \gb_2 D_2 \cdot C_2 < 0;$$ 
			and on the other hand, 
			$$\gb_1 D_1 \cdot C_1  + \gb_2 D_2 \cdot C_2 = (p_1^*D_1 + p_2^*D_2) \cdot \ti{C} = D \cdot \ti{C} \ge 0.$$ 
			This is a contradiction. 
		\end{proof}
		
		For the third lemma, we fix a point $b \in B$. 
		
		\begin{lem}\label{lem-fib}
			Fix $i=1,2$. Then there exists $N_i\in\RR$ such that for any real number $n\ge N_i$, the divisor $D_i + n \phi_i^*\OO_B(b)$ is nef.
		\end{lem}
		
		\begin{proof}
			By symmetry, we can focus on $i=2$. Let $C_1$ be a curve in $W_1$ such that $\phi_1(C_1) = B$.
			Set
			$$N_2 \cnec \frac{D_1 \cdot C_1}{ \deg (C_1 \xto{\phi_1} B)},$$
			and consider the following classes
			$$D_1' \cnec D_1 - N_2 \phi_1^*\OO_B(b) \ \ \text{ and } \ \ D_2' \cnec D_2 + N_2 \phi_2^*\OO_B(b).$$
			By construction, we have $D'_1 \cdot C_1 = 0\mbox{ and }D = p_1^*D_1' + p_2^*D_2'.$
			
			We want to show that $D_2'$ is nef. Let $C_2$ be a curve in $W_2$. If it is contained in a fiber of $\phi_2$, then $D'_2 \cdot C_2 \ge 0$ by Lemma~\ref{lem-vert}.
			Suppose now that $\phi_2(C_2) = B$, let   $\ti{C}$ be a curve in the fiber product $C_1 \times_{B} C_2$ dominating $B$, and define $\gb_1,\gb_2 \in \ZZ_{>0}$
			such that $p_{i*}\ti{C} = \gb_i C_i$.
			We have
			\begin{align*}
				\gb_2 D'_2 \cdot C_2   
				&= \gb_1 D'_1 \cdot C_1  + \gb_2 D'_2 \cdot C_2 \\
				&= (p_1^*D'_1 + p_2^*D'_2) \cdot \ti{C}\\ 
				&= D  \cdot \ti{C} \ge 0.
			\end{align*}
			So $D'_2$ is nef. A fortiori, for $n\ge N_2$, the following class 
			$$D_2+n\phi_2^*\OO_B(b) = D'_2+(n - N_2)\phi_2^*\OO_B(b)$$
			is also nef. 
		\end{proof}
		
		Let us resume the proof of Theorem \ref{thm-nefdec}. For any $t \in \RR$, let 
		$$D_1(t) \cnec D_1 - t \phi_1^*\OO_B(b) \ \ \text{ and } \ \ D_2(t) \cnec D_2 + t \phi_2^*\OO_B(b).$$ 
		By Lemma~\ref{lem-fib}, we can define intervals
		$$I_1 = ]-\infty,-N_{1,\mathrm{min}}] \ \ \text{ and } \ \ I_2 = [N_{2,\mathrm{min}}, +\infty[$$
		such that $D_i(t)$ is nef if and only if $t \in I_i$.
		Since we have for all $t\in\RR$,
		$$		D = p_1^*D_1(t) + p_2^*D_2(t),$$
		Lemma~\ref{lem-nef} shows that either $D_1(t)$ or $D_2(t)$ is nef, i.e., $I_1 \cup I_2 = \RR$.
		Hence, $I_1 \cap I_2$ is non-empty, and fixing an element $t$ in this intersection, both
		$D_1(t)$ and $D_2(t)$ are now nef, 
		giving a desired decomposition.
		
		The decomposition of the ample cone of $W$ finally follows from the decomposition of the nef cone by~\cite[Corollary 6.6.2]{MR0274683}.
	\end{proof}
	
	\begin{remark} In the setup of Theorem~\ref{thm-nefdec}, we also have a decomposition of the relative nef cone 
		\[ \Nef(W/B) = p_1^*\Nef(W_1/B) + p_2^*\Nef(W_2/B) 
		\]
		by the projection formula -- this is exactly Lemma \ref{lem-vert}.
	\end{remark}
	
	As a consequence of Theorem \ref{thm-nefdec}, we prove Corollary~\ref{cor-extr}.
	
	\begin{proof}[Proof of Corollary~\ref{cor-extr}]
		
		First, consider $E\in\Nef(W)$ spanning an extremal ray of $\Nef(W)$. Then by Theorem \ref{thm-nefdec}, there is a decomposition $E=p_1^*E_1+p_2^*E_2$, with $E_i\in\Nef(W_i)$, and either $E_1$ or $E_2$ is non zero. By extremality, $E$ is thus either in $p_1^*\Nef(W_1)$, or in $p_2^*\Nef(W_2)$.
		By symmetry, we can assume that $E=p_1^*D$, for some $D\in\Nef(W_1)$. Let us show that $D$ spans an extremal ray in $\Nef(W_1)$. Let $D = F + F'$ be any decomposition with
		$F, F' \in \Nef(W_1)$. Then $E = p_1^*D = p_1^*F + p_1^*F'$ with $p_1^*F, p_1^*F' \in \Nef(W)$, and thus by extremality, $p_1^*F$ and $p_1^*F'$ are proportional. Since $p_1^* : N^1(W_1)_{\RR} \to N^1(W)_{\RR}$ is injective, $F$ and $F'$ are proportional as well. This shows that $D$ spans an extremal ray.
		
		We thus know that every extremal ray of $\Nef(W)$ is obtained by pulling back an extremal ray of either $\Nef(W_1)$ or $\Nef(W_2)$. 
		
		Next assume that $D \in \Nef(W_1)$ is extremal, and let us prove that $p_1^*D$ is extremal in $\Nef(W)$. 
		Let $p_1^*D = E + E'$ be a decomposition with
		$E,E' \in \Nef(W)$. Up to adding terms to $E'$, we can assume that $E$ spans an extremal ray of $\Nef(W)$.
		By Theorem~\ref{thm-nefdec}, we can write
		$$E = p_1^*E_1 + p_2^*E_2, \ \text{ and } \ 
		E'= p_1^*E'_1 + p_2^*E'_2$$
		with $E_i,E'_i \in \Nef(W_i)$.
		As $E$ is extremal, the divisors $E$, $p_1^*E_1$ and $p_2^* E_2$ are proportional. Moreover
		$p_1^*(D-E_1-E_1')=p_2^*(E_2 + E'_2) \in \Nef(W)$.
		Hence, by the projection formula, $D-E_1-E_1'$ is nef. But $D$ is extremal in the cone $\Nef(W_1)$, so $D$, $E_1$, and $E_1'$ are proportional. In particular, $p_1^*D, p_1^*E_1, p_1^*E'_1$, and $p_2^*E_2$ are all proportional, which shows that $E$ and $E'$ are proportional, and thus concludes the proof.
	\end{proof}

	We now construct various fiber products showing that Theorem~\ref{thm-nefdec} fails in general over bases $B$ of dimension at least 2. Two types of constructions are provided: In Example \ref{ex-nondecomp}, the surjective maps $\phi_1$ and $\phi_2$ are birational morphisms; in Example \ref{ex-nondecompflat}, they are smooth fibrations. The first construction simply involves $(-1)$-curves on blow-ups of $\PP^2$; the second construction uses Serre's construction of vector bundles of rank two.
	
	\begin{example}\label{ex-nondecomp}
		
		Take $S:=\PP^2$, and take four points $P_1$, $P_2$, $P_3$, $P_4$ in $S$ so that no three of them lie on a line. Let $\ell_1$ be the line through $P_1$, $P_2$, and let $\ell_2$ be the line through $P_3$, $P_4$. 
		Take 
		
		$$
		W_1:=\mathrm{Bl}_{P_{1}, P_{2}}(S) \ \ 
		\text{ and } \ \ W_2:=\mathrm{Bl}_{P_{3}, P_{4}}(S).$$ 
		We let $$W:=W_1\times_{S} W_2.$$ As the blown-up points are distinct, $W$ is isomorphic to $\mathrm{Bl}_{P_1,P_2,P_3,P_4}(S)$, which is smooth. Moreover, the decomposition of the Picard group 
		$$\Pic(W)=p_1^*\Pic(W_1)+p_2^*\Pic(W_2)$$
		clearly holds.
		
		Denote by $\ell_1'$ and $\ell_2'$ the strict transforms of $\ell_1$ and $\ell_2$ in $W_1$ and $W_2$ respectively. Then $\ell_i'$ is an effective non-nef divisor on $W_i$ as $(\ell_i')^2=-1$. 
		Let $$D:=p_1^*\ell_1'+p_2^*\ell_2'.$$ 
		We show that $D$ is nef;
		this also shows that Lemma~\ref{lem-nef} fails when $\dim B \ge 2$.
		As $D$ is effective, it is enough to check that its intersections with its components are all non-negative. By symmetry, it is enough to compute $$D\cdot p_1^*\ell_1'=(\ell_1')^2+\ell_2'\cdot \phi_2^*\ell_1=-1+1=0.$$ 
		So $D$ is nef, and has vanishing intersection with the curves $p_1^*\ell_1'$ and $p_2^*\ell_2'$.
		
		Now assume by contradiction that $D$ has another decomposition $D=p_1^*D_1+p_2^*D_2$ with $D_i \in \Nef(W_i)$. Then 
		we have
		$$p_1^*(\ell'_1 - D_1)
		= p_2^*(D_2 - \ell'_2).$$
		As $p_1^*N^1(W_1)_{\RR} \cap p_2^*N^1(W_2)_{\RR}$ clearly has dimension one, it
		equals $\RR[p^*\OO_{\PP^2}(1)]$,
		where $p$ is the natural projection $W \to S$.
		Hence, for some $c\in\RR$, we have
		$$p_1^*(\ell'_1 - D_1)
		= p_2^*(D_2 - \ell'_2) = c p^*\OO_{\PP^2}(1).$$
		
		Since
		$$p_1^*D_1\cdot p_i^*\ell_i'+p_2^*D_2\cdot p_i^*\ell_i'=D\cdot p_i^*\ell'_i = 0,$$ 
		and both $p_1^*D_1$ and $p_2^*D_2$ are nef, 
		we have $p_i^*D_i\cdot p_i^*\ell_i'=0$.
		Thus
		$$-1 = p_1^*\ell_1' \cdot p_1^*(\ell'_1 - D_1) 
		= c p_1^*\ell_1' \cdot p^*\OO_{\PP^2}(1) = c$$
		and similarly, 
		$$1 = p_2^*\ell_2' \cdot p_2^*(D_2 - \ell'_2)
		= c p_2^*\ell_2' \cdot p^*\OO_{\PP^2}(1) = c,$$
		which is a contradiction.
	\end{example}
	
	\begin{example}\label{ex-nondecompflat}
		Take $S\cnec\PP^2$. Let us fix a closed subscheme $Z_2$ of $\PP^2$ consisting of two distinct (reduced) points. We fix another closed subscheme $Z_1$ of $\PP^2$ consisting of two distinct (reduced) points, chosen generally with respect to $Z_2$.
		
		For each $i = 1,2$, Serre's construction
		(see, e.g.,~\cite[Theorem 5.1.1]{HuybLehn}) 
		produces a locally free sheaf $E_i$ of rank $2$ on $\PP^2$, which fits into the short exact sequence
		\begin{equation}\label{suitex-Ei}
			0\to\OO_{\PP^2}\to E_i \to {\mathcal I}_{Z_i}(i)\to 0.  
		\end{equation}
		
		Set $W_i \cnec \PP(E_i)$; see \cite[Definition in p.162]{HarBook}.
		Consider 
		\[ \xymatrix@=1.5em{ & W \coloneqq W_1 \times_{S} W_2 \ar[dl]^{p_1} \ar[dd]^{p} \ar[dr]_{p_2}   \\
			W_1 = \PP(E_1) \ar[dr]_{\phi_1} & &     W_2 = \PP(E_2)\ . \ar[dl]^{\phi_2}   \\
			& S = \PP^2 } \]
		As a projectivized vector bundle, each $W_i$ is endowed with a tautological line bundle $\zeta_i$ satisfying
		${\phi_i}_*\zeta_i = E_i$. In particular, this line bundle has a distinguished section given by the inclusion morphism in (\ref{suitex-Ei}), whose zero locus we denote by $S_i$. We will describe the geometry of $S_i$ later.
		
		Note that the N\'eron--Severi space of $W$ decomposes. Indeed, the smooth fibration $p_1 : W \to W_1$ identifies with the projectivization of the vector bundle $\phi_1^*E_2$ over $W_1$, which has tautological line bundle $p_2^*\zeta_2$, so
		\begin{align*}
			N^1(W)_{\RR} &= p_1^*N^1(W_1)_{\RR}+ \RR \cdot p_2^*[\zeta_2]\\
			&= \RR\cdot p_1^*[\zeta_1] + p^*N^1(S)_{\RR} + \RR \cdot p_2^*[\zeta_2]\\
			&= p_1^*N^1(W_1)+p_2^*N^1(W_2).
		\end{align*}
		
		Define the line bundle $$D=p_1^*\zeta_1+p_2^*\zeta_2$$ 
		on $W$.
		It is effective, as the $\zeta_i$ both are. To prove that $D$ is nef, let us describe the geometry of the zero loci $S_i$.
		
		By \cite[Proposition 3.6.2]{EGAII}, 
		and since the closed subschemes $Z_i$ are locally complete intersections, each zero locus
		$S_i$ is in fact a (reduced irreducible) surface, isomorphic to $\mathrm{Bl}_{Z_i}\PP^2\simeq \PP({\mathcal I}_{Z_i}(i))$ naturally embedded in $\PP(E_i)$ through the surjection in (\ref{suitex-Ei}). Through this identification, the restricted line bundle ${\zeta_i}_{|S_i}$ corresponds to the tautological line bundle of $\PP({\mathcal I}_{Z_i}(i))$, which in $\mathrm{Bl}_{Z_i}\PP^2$ corresponds to the dual of the exceptional line bundle twisted by $\phi_i^{\ast}\OO_{\PP^2}(i)$.
		For $i=1$, this line bundle corresponds to the divisor obtained by strict transform of the line $\ell_1$ passing through the two points of $Z_1$ with the following properties:
		\begin{itemize}
			\item It is effective and has a unique section which is irreducible; 
			
			\item It has negative square.
		\end{itemize}
		For $i=2$, it is the strict transform of any conic through the two points of $Z_2$ with the following properties: 
		\begin{itemize}
			\item It is effective and admits an irreducible section; 
			
			\item It has positive square.
		\end{itemize} 
		
		Let us summarize: On one hand, $\zeta_1$ has exactly one negative curve $\ell_1'$ on $W_1$, which is contained in $S_1$ and has negative square there. On the other hand, $\zeta_2$ is nef on $W_2$. 
		
		We now prove that $D$ is nef, arguing by contradiction: Assume that there is a curve $C$ in $W$ such that $D\cdot C < 0$. We just proved that $p_2^*\zeta_2$ is nef, so $\zeta_1\cdot {p_1}_*C < 0$, and so there is a positive integer $m$ such that ${p_1}_*C = m\ell_1'$; moreover, $C$ must lie in $p_1^{-1}(\ell_1')$. The restricted map $\phi_1:\ell_1'\to\ell_1$ is an isomorphism, so its base change $p_2:p_1^{-1}(\ell_1')\to \phi_2^{-1}(\ell_1)\simeq\PP({E_2}_{|\ell_1})$ is an isomorphism too. Thus, ${p_2}_*C$ is a reduced curve $C_2$ in $W_2$, and ${\phi_2}_* C_2=m\ell_1$.
		By the projection formula, 
		$$D\cdot C=m\zeta_1\cdot\ell'_1 + \zeta_2\cdot C_2.$$
		We have ${\zeta_1}_{|S_1}=\ell_1'$, so $\zeta_1\cdot\ell'_1=-1$. Moreover, by \cite[Example 1 in \S5.2, Chapter 1]{OSS80}, and since we chose $Z_1$ generally with respect to $Z_2$, we have
		\begin{equation}\label{resE_2}
			{E_2}_{|\ell_1} \simeq \cO_{\PP^1}(1) \oplus \cO_{\PP^1}(1),
		\end{equation}
		so ${E_2}_{|\ell_1}\otimes\OO_{\ell_1}(-1)$ is nef, in particular
		$\zeta_2\cdot C_2\ge \phi_2^*\OO_{\ell_1}(1)\cdot C_2 = m$. Hence, we finally have $0>D\cdot C\ge -m+m = 0$, contradiction. So $D$ is nef.
		
		We conclude this example by picking a decomposition of $D$ as $p_1^*D_1+p_2^*D_2$ with $D_i\in N^1(W_i)_{\RR}$, and proving that at least one of the $D_i$ is not nef.
		Since the intersection of $p_1^* N^1(W_1)_{\RR}$ with $p_2^* N^1(W_2)_{\RR}$ is the subspace $\RR\cdot p^*[\OO_{\PP^2}(1)]$, and since we already have $D=p_1^*\zeta_1+p_2^*\zeta_2$, there exists $a\in\RR$ such that
		$$D_1 = \zeta_1 + a \phi_1^*\OO_{\PP^2}(1), \ \ 
		D_2 = \zeta_2 - a \phi_2^*\OO_{\PP^2}(1).$$
		In particular, $$D_1 \cdot \ell_1' = \zeta_1\cdot\ell_1' + a \OO_{\PP^2}(1)\cdot{\phi_1}_*\ell_1'= - 1 + a.$$
		Moreover, by \cite[Example 1 in \S5.2, Chapter 1]{OSS80} again, there exists a line $\ell_2$ in $\PP^2$ such that 
		\begin{equation}\label{resE_2_2}
			{E_2}_{|\ell_2} \simeq \cO_{\PP^1}\oplus \cO_{\PP^1}(2),
		\end{equation}
		and so there is a section $\ell_2'$ of the fibration $\phi_2:\PP({E_2}_{|\ell_2})\to\ell_2$ such that $\zeta_2\cdot\ell_2'=0$. In particular,
		$$D_2 \cdot \ell_2' =
		\zeta_2\cdot\ell_2' - a \ell_2\cdot{\phi_2}_*\ell_2' = - a.
		$$
		Since at least one of the two numbers $a-1$ and $-a$ is negative, $D_1$ and $D_2$ cannot both be nef.
	\end{example}
	
	We now use Examples~\ref{ex-nondecomp} and \ref{ex-nondecompflat} to build similar counter-examples over bases of higher dimension.
	
	\begin{example}\label{ex-nondecompbis}

		Take $W$, $W_1$, $W_2$ and $S$ as in Example \ref{ex-nondecomp} or Example \ref{ex-nondecompflat}. 
		Note that they all are rationally connected: It is clear in Example \ref{ex-nondecomp}, and follows from \cite[Corollary 1.3]{GHS03} in Example  \ref{ex-nondecompflat}.
		Introduce
		$$W \times T = (W_1 \times T) \times_{(S \times T)} (W_2 \times T)$$
		where $T$ is an arbitrary smooth projective variety. 
		Since $W$, $W_1$ and $W_2$ are rationally connected and smooth, they have trivial irregularity, so that
		$$N^1(Z \times T)_{\RR} =  p_Z^*N^1(Z)_{\RR}\oplus p_T^*N^1(T)_{\RR},$$
		for $Z=W$, $W_1$ or $W_2$. 
		This implies that
		$$N^1(W \times T)_{\RR} = 
		(p_1 \times \id_T)^*N^1(W_1 \times T)_{\RR} +
		(p_2 \times \id_T)^*N^1(W_2 \times T)_{\RR}.$$
		Note that by the projection formula, 
		$$\Nef(Z \times T) =  p_Z^*\Nef(Z)\oplus p_T^*\Nef(T),$$
		for $Z=W$, $W_1$ or $W_2$. So, if we assume by contradiction that 
		$$\Nef(W\times T)=(p_1\times\mathrm{id}_T)^*\Nef(W_1\times T)+(p_2\times\mathrm{id}_T)^*\Nef(W_2\times T),$$
		we get $\Nef(W)=p_1^*\Nef(W_1)+p_2^*\Nef(W_2)$, which contradicts Example~\ref{ex-nondecomp} or  Example~\ref{ex-nondecompflat}.
	\end{example}

	\begin{remark}\label{rem-exMov} 
		We note that 
		Theorem~\ref{thm-nefdec} 
		also fails if the nef cones are replaced by 
		the movable cones. 
		In general, let $X$ be a smooth projective variety and recall that a divisor $D$ on $X$ is called \textit{movable}, if there is a positive integer $m$ such that $mD$ is effective and the base locus of the linear system $|\cO_X(mD)|$ has no component of codimension 1. The closed movable cone $\Mov(X)$
		is then defined as the closure of the convex cone in $N^1(X)_{\RR}$ generated by the classes
		of movable divisors. It always holds $\Nef(X) \subset \Mov(X)$, and, if moreover $X$ is a surface, then
		$\Nef(X) = \Mov(X)$.
		
		Take a general fiber product $W = W_1 \times_{\PP^1} W_2$ of two very general rational elliptic surfaces $W_1 \to \PP^1$ and $W_2 \to \PP^1$ with sections.
		Then $W$ is a strict Calabi--Yau threefold and has non-trivial algebraic flops (see~\cite{Na91}). Thus, $\Nef(W) \subsetneq \Mov(W)$.
		But since the $W_i$ are surfaces, we have 
		$$p_1^*\Mov(W_1) + p_2^*\Mov(W_2) = p_1^*\Nef(W_1) + p_2^*\Nef(W_2) =\Nef(W) \subsetneq \Mov(W),$$
		where the second equality follows from Theorem \ref{thm-nefdec}.
		
		Even in this particular case, the version of the Cone Conjecture stated in \cite[Conjecture 2.1.(2)]{To10} is not known.
	\end{remark}
	
	We conclude this section with a corollary of Theorem~\ref{thm-nefdec} that will be key in the proof of Theorem \ref{main}.
	
	For a morphism $\pi:X\to Y$, we define $$\Aut(X/Y)=\{g\in\Aut(X)\mid \pi\circ g = \pi\}.$$
	
	\begin{cor}\label{cor-H1H2} For $i=1,2$, let $\phi_i : W_i \to B$ be a surjective morphism from a projective variety to a projective curve $B$; let $H_i$ be a subgroup of $ \Aut(W_i/B)$. Assume that
		\begin{enumerate}
			\item  The fiber product $W=W_1 \times_{B} W_2$ is irreducible;
			\item  It holds $$p_1^*N^1(W_1)_{\RR}+p_2^*N^1(W_2)_{\RR}=N^1(W)_{\RR},$$ where $p_i$ denotes the projection from $W$ onto $W_i$;
			\item For each $i=1,2$, there exists a rational polyhedral cone $\Pi_i$ in $\Nefp(W_i)$ such that $\Amp(W_i)\subset H_i\cdot \Pi_i$.
		\end{enumerate}
		Then, for any subgroup $H$ of $\Aut(W)$ containing $H_1\times H_2$, there is a rational polyhedral fundamental domain for the $H$-action on $\Nefp(W)$.
	\end{cor}
	
	\begin{proof}
		Let $\Pi$ be the convex hull of $p_1^*\Pi_1 + p_2^*\Pi_2$. Then $\Pi$ is a rational polyhedral cone contained in $\Nef^+(W)$. Moreover, 
		\[ \Amp(W) \subset(H_1 \times H_2) \cdot \Pi \subset H \cdot \Pi
		\]
		as $p_1^*\Amp(W_1) + p_2^*\Amp(W_2)=\Amp(W)$ by Theorem~\ref{thm-nefdec}. The existence of a rational polyhedral fundamental domain then follows from Proposition~\ref{pro-looij}.(1).
	\end{proof}

	\section{Construction of Schoen varieties}\label{cons}
	
	Schoen varieties are constructed as fiber products
	of two fibrations over $\PP^1$.
	Let us first construct these fibrations.
	
	\subsection{The factor $W$ with a fibration over $\PP^1$}\label{subsec-Wconstr} 
	
	This construction relies on a pencil of ample hypersurfaces in a Fano manifold.
	
	Let $Z$ be a Fano manifold of dimension at least $2$, and let $D$ be an ample divisor in $Z$ such that
	both $\cO_Z(D)$ and $\cO_Z(-K_Z - D)$ are globally generated.
	Note that
	$\cO_Z(-K_Z)$ is then globally generated as well.
	
	\begin{example}
		Take any toric Fano manifold $Z$ of dimension at least $2$. 
		Since nef line bundles on a projective toric manifold are globally generated,
		any decomposition $-K_Z = D + D'$ as the sum of 
		an ample divisor $D$ and a nef divisor $D'$ yields a pair $(Z,D)$ satisfying the above condition.
	\end{example}

	Let $W \subset \PP^1\times Z$ be a general member of 
	the ample and basepoint-free linear system $|\OO_{\PP^1}(1)\boxtimes\OO_{Z}(D)|$. 
	We have a fibration $\phi : W \to \PP^1$ via the first projection, and the second projection $\varepsilon : W \to Z$ 
	is the blow-up of $Z$ 
	along the smooth subvariety $Y$ of codimension two cut out by the members of the pencil in $|D|$ defined by $W$.
	Since $Z$ is Fano, $W$ is rationally connected.
	By construction, any point $y\in Y$ defines a rational curve $\gep^{-1}(y)$ which is a section of $\phi : W \to \PP^1$.

	By the adjunction formula, 
	\begin{equation}\label{acd} 
		\OO_W(-K_{W}) = \left(\OO_{\PP^1}(1)\boxtimes \OO_{Z}(-K_Z - D)\right)|_{W},
	\end{equation}
	so $\OO_W(-K_W)$
	is globally generated, {\it a fortiori} nef and effective.

	If $Z$ is chosen to be a del Pezzo surface, then the surface $W$ is described by the following lemma. Recall that a smooth projective surface $S$ is called \textit{weak del Pezzo} if its anticanonical divisor $-K_S$ is nef and big. 
	
	\begin{lemma}\label{lem-dimZ2} If $Z$ has dimension 2, then either $D \in |-K_Z|$ and $W \xrightarrow{\phi} \PP^1$ is a rational elliptic surface with globally generated anticanonical line bundle, or $W$ is a weak del Pezzo surface. 
	\end{lemma}
	
	\begin{proof} 
		Since $W$ is rationally connected and $\dim W = 2$, we know that $W$ is rational.
		If $D \in |-K_Z|$,
		then $\cO_W(-K_{W}) = \phi^*\cO_{\PP^1}(1)$, which is globally generated, and which makes
		$W$ into a rational elliptic surface.
		
		Suppose now that $D \notin |-K_Z|$.
		As $-K_Z - D$ is effective and non-trivial, and as $-K_Z$ and $D$ are ample, we have
		$-K_Z(-K_Z - D) > 0$ and $D(-K_Z - D) > 0$, and thus, 
		$$K_Z^2 > -K_Z\cdot D > D^2. $$
		As $W$ is the blowup of $Z$ at $(D^2)$ points,
		we have $K_W^2 = K_Z^2 - D^2 > 0$.
		Since $-K_W$ is nef, 
		$W$ is 
		a weak del Pezzo surface. 
	\end{proof}
	
	\begin{remark}\label{Wgeneral}
		Note that, in the case where $W$ is a rational elliptic surface, the fact that it has a section and that it is chosen general in its pencil on $\PP^1\times Z$ implies that it is isomorphic to $\PP^2$ blown-up in the base locus of a general pencil of cubics. In particular, $W$ has topological Euler characteristics $12$, the canonical fibration $W\to\PP^1$ has some singular fibers, but no multiple fibers.
		The fact that the rational elliptic surface $W$ is general implies that the singular fibers of $W\to\PP^1$ are exactly 12 nodal rational curves (\cite[p.8]{Mi89}).
		
		Considering the $j$-invariant in family for the fibration $W\to\PP^1$, we obtain a proper surjective map $j:\PP^1\to\PP^1$ which is finite of degree 12, and has 12 simple poles which occur at the 12 image points of the 12 singular fibers (\cite[Lemma (IV.4.1), Corollary (IV.4.2)]{Mi89}).
	\end{remark}
	
	In general, the construction of $W$ described above ensures the following properties.
	
	\begin{proposition}\label{prop-lef} 
		We have
		$$\Nefe(W) = \Nefp(W) = \Nef(W).$$
		Moreover, if $\dim W \geq 3$, or if
		$W$ is a weak del Pezzo surface, then the cone $\Nef(W)$ is rational polyhedral, spanned by classes of semiample divisors.
	\end{proposition}
	
	\begin{proof}
		
		We start with the ``moreover'' part. It is a corollary of some known results.
		If $W$ is a weak del Pezzo surface,
		then $W$ is log Fano (see e.g.~\cite[Proposition 2.6]{Mass}). 
		Hence by the Cone Theorem \cite[Theorem 3.7]{KM98}, its nef cone is a rational polyhedral cone spanned by classes of semiample divisors.
		Assume that $\dim W \geq 3$.
		Since $\PP^1 \times Z$ is a smooth Fano variety of dimension at least four,
		and since $W \subset \PP^1 \times Z$ is a smooth ample divisor such that
		$$\cO_{\PP^1 \times Z}(-K_{\PP^1 \times Z} - W) = \cO_Z(-K_Z - D) \boxtimes \cO_{\PP^1}(1)$$ 
		is nef, we can apply~\cite[Proposition 3.5]{BI09} (which generalizes~\cite[Appendix]{Ko91}). It yields an isomorphism
		$$j_*: \NE(W) \eto \NE(\PP^1 \times Z)$$
		induced by the inclusion $j : W \hto \PP^1 \times Z$.
		Dually, we obtain an isomorphism
		$$j^*: \Nef(\PP^1 \times Z) \eto \Nef(W).$$
		As $\Nef(\PP^1 \times Z)$ is rational polyhedral and spanned by classes of semiample divisors, so is $\Nef(W)$.
		
		We now prove the equality of the three cones $\Nefe(W)$, $\Nefp(W)$, and $\Nef(W)$. If $\dim W\ge 3$, or if $W$ is a weak del Pezzo surface, the equality clearly follows from the fact that $\Nef(W)$ 
		is rational polyhedral, spanned by classes of semiample divisors. So by Lemma~\ref{lem-dimZ2}, we can focus on the case where $W$ is a rational elliptic surface. 
		
		Clearly, $\Nefe(W)$ and $\Nefp(W)$ are subcones of $\Nef(W)$. Moreover, $\Nefp(W)\subset\Nefe(W)$ by~\cite[Lemma 4.2]{To10}. We only need to show that
		$\Nef(W) = \Nefp(W)$. By~\cite[Corollary 3.3.(c)]{Nikulin}, the cone $\NE(W)$ is generated by curve classes, so dually, $\Nef(W)$ is spanned by Cartier divisors. So $\Nef(W) = \Nefp(W)$ indeed.
	\end{proof} 
	
	Let us conclude the description of $W$ by describing the general fiber of $\phi : W\to\PP^1$, under the assumption that $D \in |-K_Z|$.
	
	\begin{lemma}\label{lem-CYfiber}
		Suppose that $D \in |-K_Z|$.
		Then the general fiber $F$ of $\phi:W\to\PP^1$ is a Calabi--Yau manifold (as in Definition \ref{def-cymfd}).
	\end{lemma}
	
	\begin{proof} 
		Since $D \in |-K_Z|$, the general fiber $F$ is linearly equivalent to the anticanonical divisor $-K_W$ by~\eqref{acd}.
		By adjunction, $F$ has trivial canonical bundle. We also have an exact sequence
		$$0\to \OO_W(-K_W) \to\OO_W\to \OO_F \to 0.$$
		Since $W$ is rationally connected,
		we have 
		$$h^{\dim W - i}(W,-K_W) = h^i(W,\OO_W)=0$$ 
		for $i \ge 1$.
		Hence $h^i(F,\OO_F)=0$ whenever $1 \le i \le \dim W - 2 = \dim F - 1 $.
	\end{proof}
	
	\subsection{The fiber product $X=W_1{\times}_{\PP^1} W_2$}\label{subsec-Xconstr}
	
	We are ready to generalize Schoen's construction and obtain Calabi--Yau pairs in arbitrary dimension. 
	For $i = 1, 2$, let $Z_i, D_i, W_i$ be as in \S\ref{subsec-Wconstr}. We denote by $\phi_i:W_i\to \PP^1$ the associated fibration, and recall that it has a section.
	
	We add one assumption, which is automatically satisfied by taking the fibrations $\phi_i$ for $i=1,2$ to be general with respect to one another: 
	
	\medskip 
	
	\begin{center}
		\textit{For every $t\in\PP^1$, the fiber of at least one of the $\phi_i$ above $t$ is smooth.} 
	\end{center}
	
	\medskip \noindent In the case where both $W_i$ are rational elliptic surfaces, this assumption has an important consequence.
	
	\begin{lemma}\label{nonisogenous}
		Let $W_1$ and $W_2$ be general rational elliptic surfaces, with their canonical fibrations $\phi_i:W_i\to\PP^1$, each admitting a section. Assume that for every $t\in\PP^1$, there is $i$ such that the fiber $\phi_i^{-1}(t)$ is smooth.
		Then, for a very general point $t\in\PP^1$, the fibers $\phi_1^{-1}(t)$ and $\phi_2^{-1}(t)$ are smooth, non-isogenous elliptic curves.
	\end{lemma}

	\begin{proof} For $i=1,2$, consider the finite morphism $j_i:\PP^1\to\PP^1$ induced by the $j$-invariant of the elliptic fibration $\phi_i$ (see Remark \ref{Wgeneral}). Define the morphism $J\cnec (j_1,j_2):\PP^1\to\PP^1\times\PP^1$. Its image is an irreducible curve in $\PP^1 \times \PP^1$. 
		
		For each positive integer $n$, let $F_n(x, y) \in \ZZ[x, y]$ be the polynomial as in \cite[Theorem 6.3 in p.146]{Si94}. Then by \cite[Exercise 2.19.(a) in p.182]{Si94}, we have $F_n(j(\phi_1^{-1}(t)), j(\phi_2^{-1}(t))) = 0$ if and only if there is an isogeny $\phi_1^{-1}(t) \to \phi_2^{-1}(t)$ of degree $n$. By \cite[Exercise 2.18.(e) in p.181]{Si94}, each $F_n(x, y)$ is a product of some polynomials $\Phi_m(x, y)$ indexed by positive integers. By the expression in \cite[Exercise 2.18 in p.181]{Si94}, each $\Phi_m(x, y)$ viewed as a polynomial in the single variable $x$ has leading coefficient $1$. Together with \cite[Exercise 2.18.(b) in p.181]{Si94}, we obtain the irreducibility of $\Phi_m(x, y)$ in $\CC[x, y]$. 
		
		Let $\Sigma_m \subset \PP^1 \times \PP^1$ be the irreducible curve defined by the homogenization of $\Phi_m(x, y)$ using $x = s/t$ and $y = u/v$.  We claim that $J(\PP^1)$ intersects with each $\Sigma_m$ at finitely many points. Indeed, by our assumption, we can take $t \in \PP^1$ such that the fiber $\phi_1^{-1}(t)$ is singular, while the fiber $\phi_2^{-1}(t)$ is an elliptic curve, so $J(t)=([1:0],[\alpha:1])$ for some $\alpha\in\CC$. As we mentioned before, each $\Phi_m(x, y)$ viewed as a polynomial in the single variable $x$ has leading coefficient $1$, so $([1 : 0], [\alpha : 1])\notin \Sigma_m$. This implies $J(\PP^1) \neq \Sigma_m$, and the claim holds because both $\Sigma_m$ and $J(\PP^1)$ are irreducible. 
		
		Let $\Sigma \subset \PP^1 \times \PP^1$ be the union of the countably many curves $\Sigma_m$. Then the set $\PP^1\setminus (Z_1\cup Z_2\cup J^{-1}(J(\PP^1)\cap\Sigma))$ is non-empty with the property that each of its elements is a very general point, say $t$, satisfying that the fibers $\phi_1^{-1}(t)$ and $\phi_2^{-1}(t)$ are smooth, non-isogenous elliptic curves. 
	\end{proof}

	Now that we better understand the fibrations $\phi_i$ relatively to one another, we can consider the fiber product over $\PP^1$
	\[ \xymatrix@=1.5em{ & X = W_1 \times_{\PP^1} W_2 \ar[dl]_{p_1} \ar[dd]_{\phi} \ar[dr]^{p_2}   \\
		W_1 \ar[dr]_{\phi_1} & &     W_2 \ . \ar[dl]^{\phi_2}   \\
		& \PP^1 } \]
	As for every $t\in\PP^1$, the fiber of at least one of the $\phi_i$ above $t$ is smooth, the variety $X$ is smooth too. 
	We can also view $X$ as a complete intersection of two hypersurfaces in $\PP^1 \times Z_1 \times Z_2$, given by general members in the linear systems
	\[ |\OO_{\PP^1}(1)\boxtimes\OO_{Z_1}(D_1)\boxtimes \OO_{Z_2}| \ \  \mathrm{and} \ \ |\OO_{\PP^1}(1)\boxtimes\OO_{Z_1}\boxtimes \OO_{Z_2}(D_2)|. \]
	By adjunction, we obtain that
	\begin{equation}\label{-KX}
		\cO_X(-K_X) = \left(\OO_{\PP^1}\boxtimes\OO_{Z_1}(-K_{Z_1} - D_1) \boxtimes \OO_{Z_2}(-K_{Z_2} - D_2)\right)|_X,
	\end{equation}
	which is globally generated, hence nef and effective.
	
	\begin{definition}\label{def_sch}
		A smooth projective variety $X$ constructed as above is called a \textit{Schoen variety}. 
		A pair $(X, \Delta)$ is called a \textit{Schoen pair} if $X$ is a Schoen variety, and $\Delta$ is an effective $\QQ$-divisor such that $K_X + \Delta \sim_{\QQ} 0$. 
	\end{definition}
	
	Any Schoen variety $X$ can be associated many Schoen pairs $(X, \Delta)$ as long as $-K_X$ is non-trivial. Every Schoen pair is by definition a Calabi--Yau pair (as in Definition \ref{def-cypair}). Moreover, if $(X,\Delta)$ is a Schoen pair, then there exists a positive integer $m$ such that 
	\begin{equation}\label{eq_sp} 
		\Delta = \frac{1}{m}\Delta_{m,X}, \ \text{ with } \,\Delta_{m,X}\in|-mK_X|. 
	\end{equation}
	If $m \ge 2$ and $\Delta_{m,X}\in|-mK_X|$ is general, the Calabi--Yau pair $(X,\Delta)$ is klt.

	To conclude this section, we prove that, if for both $i=1,2$, the divisor $D_i$ chosen when constructing $W_i$ is in the linear system $|-K_{Z_i}|$, then the Schoen variety $X$ is a strict Calabi--Yau manifold.
	
	\begin{lemma}\label{su2.1} 
		Any Schoen variety $X$ is simply connected.
	\end{lemma}
	
	\begin{proof} The proof is similar to \cite[Lemma 1]{Sc86} and \cite[Lemma 2.1]{Su21}. 
		
		Let $U \subset \PP^1$ be the open subset over which the morphism $\phi: X \to \PP^1$ is smooth and set $V := \phi^{-1}(U)$. Let $i: V \hookrightarrow X$ and $j: U \hookrightarrow \PP^1$ be the natural inclusions. The restriction $\phi' \coloneqq \phi|_V : V \to U$ is topologically locally trivial with a fiber, say $F$. Since both $\phi_1$ and $\phi_2$ have sections, $\phi: X \to \PP^1$ also admits a section $\gs : \PP^1 \to X$. Consider the commutative diagram
		\[ \xymatrix{
			1 \ar[r]^{} &\pi_1(F)  \ar[r]^{} & \pi_1(V) \ar@{>>}[d]_{i_{\ast}} \ar@{>>}@<-3pt>[r]_{\phi'_{\ast}} & \pi_1(U) \ar@{>}[d]^{j_{\ast}} \ar@/_/[l]_{{\gs_{U}}_*} \ar[r]^{} & 1\\
			& & \pi_1(X) \ar@{>>}@<-3pt>[r]_{\phi_{\ast}} & \pi_1(\PP^1). \ar@/_/[l]_{{\gs}_*}  }
		\]
		Here the first row is exact by the homotopy long exact sequence, and $i_{\ast} : \pi_1(V) \to \pi_1(X)$ is surjective by \cite[Proposition 2.10.1]{Ko95}. 
		
		We claim that the image of $\pi_1(F)$ in $\pi_1(X)$ equals $\pi_1(X)$. Indeed, since $\pi_1(\PP^1)$ is trivial, the composition
		$i_* \circ {\gs_U}_* = \gs_* \circ j_*$
		is trivial.
		Using that $i_*$ is surjective, that this composition is trivial, and that $\pi_1(V)$ is generated by the union of its subgroups $ \pi_1(F)$ and ${\gs_U}_*\pi_1(U)$, we obtain
		$$\pi_1(X) = i_*\pi_1(V)
		= i_* \pi_1(F).$$
	
	We are now left to show that the image of $\pi_1(F)$ in $\pi_1(X)$ is trivial. Write $F = F_1 \times F_2$, where $F_i$ is a general fiber of $\phi_i: W_i \to \PP^1$ for $i = 1,2$. Since
	$\pi_1(F) = \pi_1(F_1) \times \pi_1(F_2)$, 
	it is enough to show that the image of $\pi_1(F_i)$ in $\pi_1(X)$ is trivial, which we prove for $i = 1$. 
	
	A section of $\phi_{2} : W_2 \to \PP^1$ gives rise to a section
	$s$ of $p_1: X \to W_1$.
	By construction, the homomorphism $\pi_1(F_1) \to \pi_1(X)$ is induced by
	$F_1 \hto W_1 \xto{s} X$,
	thus factors through $\pi_1(W_1)$.
	Since it is rationally connected, $W_1$ is simply connected, and hence the image of $\pi_1(F_1)$ in $\pi_1(X)$ is trivial.
\end{proof}

\begin{pro}\label{su2.1bis} 
	Suppose that $D_i \in |-K_{Z_i}|$ for both $i = 1,2$.
	Then the Schoen variety $X$ is a strict Calabi--Yau manifold (see Definition~\ref{def-cymfd}). 
\end{pro}

\begin{proof}
	By ~\eqref{-KX} and Lemma \ref{su2.1}, $X$ has trivial canonical bundle, and trivial fundamental group. We are left showing that $h^p(X,\OO_X)=0$ for every $0 < p <\dim X$.
	
	\begin{lem}\label{lem-Riw} 
		Let $g: \sX \to \sY$ be a surjective morphism between smooth projective varieties. Assume that a general fiber $F$ of $g$ is a Calabi--Yau manifold and that the canonical line bundle $\omega_{\sX}$ is trivial. 
		Then, for every positive integer $q$, we have
		\begin{equation*}
			R^q g_*\OO_{\sX} = 
			\begin{cases}
				\go_\sY, \text{ if } q = \dim \sX - \dim \sY, \\
				0, \text{  \ \ otherwise.}
			\end{cases}
		\end{equation*}
	\end{lem}
	\begin{proof}
		Set
		$r \cnec \dim \sX - \dim \sY$. 
		By~\cite[Theorem 2.1.(i)]{KollDirIm} and~\cite[Corollary 3.9]{KollDirIm2}, the sheaf 
		$R^qg_*\go_{\sX}=R^qg_*\cO_{\sX}$ is reflexive. 
		Since $\sY$ is smooth, the invertibility of $R^qg_*\cO_{\sX}$ follows provided it has rank one. Its rank is explicitly given by the dimension of $H^q(F,\cO_F)$, which is one if $q=0$ or $r$, and zero otherwise. Hence, we have 
		\begin{equation}\label{higherpushf}
			R^qg_*\cO_{\sX} = 
			\begin{cases}
				\text{a line bundle, \ if } q = 0 \text{ or } r,\\
				0, \text{\ otherwise.}
			\end{cases}
		\end{equation}
		By Grothendieck--Verdier duality ~\cite[Theorem 3.34]{HuybrechtsFM}, we have
		$$Rg_* \OO_{\sX} \simeq Rg_* \go_{\sX} \simeq R\cHom(Rg_*\cO_{\sX},\go_{\sY}[-r]).$$
		The Grothendieck spectral sequence gives
		$$E_2^{p,-q} \cnec \cExt^p(R^qg_*\cO_{\sX}, \go_{\sY}) \Rightarrow R^{p-q + r}g_* \OO_{\sX}$$
		(see e.g.~\cite[Example 2.70.ii)]{HuybrechtsFM}).
		But by (\ref{higherpushf}), the page $E_2$ has exactly two non-zero entries, namely $E_2^{0,0} = \go_{\sY}$, and $E_2^{0,-r}$. So Lemma~\ref{lem-Riw} follows.  
	\end{proof}
	
	We return to our Schoen variety $X$.
	For $i=1,2$, we let $w_i \cnec \dim W_i$.
	By Lemma \ref{lem-CYfiber}, and as $p_2: X \to W_2$ is a base change of $\phi_1: W_1 \to \PP^1$, the general fiber of $p_2$ is a Calabi--Yau manifold. We can thus apply Lemma~\ref{lem-Riw} to $p_2$,
	and obtain that
	\begin{equation*}
		R^q{p_2}_*\go_{X} = R^q{p_2}_*\OO_{X} = 
		\begin{cases}
			\cO_{W_2}, \text{  if } q = 0, \\
			\go_{W_2}, \,\text{ if } q = \dim w_1 - 1, \\
			0, \text{\ \ \,\quad otherwise.}
		\end{cases}
	\end{equation*}
	Together with~\cite[Corollary 3.2]{KollDirIm2}, this yields
	$$h^p(X,\OO_X)=h^p(W_2,\OO_{W_2})+h^{p-w_1+1}(W_2,\go_{W_2})$$ for all $0\le p\le \dim X$.
	Since $W_2$ is rationally connected, this is zero as soon as $p\ne 0$ and $p<w_1+w_2-1 = \dim X$.
\end{proof}

\section{Application to the Cone Conjecture}\label{prof}

In this section, we prove Theorem \ref{main}. The set-up and the notations were defined in Section \ref{cons}: We consider a Schoen variety $X$, fitting in a Schoen pair $(X,\Delta)$. Let us recall the Cartesian diagram defining $X$:
\[ \xymatrix@=1.5em{ & X = W_1 \times_{\PP^1} W_2 \ar[dl]_{p_1} \ar[dd]_{\phi} \ar[dr]^{p_2}   \\
	W_1 \ar[dr]_{\phi_1} & &     W_2 \ . \ar[dl]^{\phi_2}   \\
	& \PP^1 } \]

\begin{lemma}\label{na1.1} 
	We have    $$p_1^*N^1(W_1)_{\RR}+p_2^*N^1(W_2)_{\RR}=N^1(X)_{\RR}.$$ 
\end{lemma}

\begin{proof}
	
	Let $p \in \PP^1$ be a very
	general point and let 
	$F_i \cnec \phi_i^{-1}(p) \subset W_i$.
	
	\begin{claim}\label{claim-isoprod}
		The map
		$$\Psi : \Pic(F_1) \times \Pic(F_2) \to \Pic(F_1 \times F_2)$$
		defined by $\Psi(L,M) = L \boxtimes M$ is an isomorphism.
	\end{claim}
	
	\begin{proof}
		First suppose that
		$W_1$ and $W_2$ are not both rational elliptic surfaces.
		If there is $i$ such that $Z_i$ has dimension at least $3$, then $F_i$ is a smooth ample hypersurface in $Z_i$, and so by Lefschetz hyperplane theorem, $F_i$ has trivial irregularity. If there is $i$ such that $Z_i$ is a surface and $D_i\notin |-K_{Z_i}|$, then $F_i$ is a smooth curve in $Z_i$, and by adjunction, it is in fact a rational curve, which again has trivial irregularity. In any case, Claim~\ref{claim-isoprod} follows 
		from~\cite[Exercise III.12.6]{HarBook}.
		
		Assume now that both $W_1$ and $W_2$ are rational elliptic surfaces. Then, by Lemma \ref{nonisogenous}, the fibers $F_1$ and $F_2$ are smooth, non-isogenous elliptic curves.
		We have a short exact sequence of abelian groups~\cite[Theorem 11.5.1]{MR2062673}
		$$0 \to \Pic(F_1) \times \Pic(F_2) \xto{\Psi} \Pic(F_1 \times F_2) \to  \Hom(F_1,F_2) \to 0,$$
		where $\Hom(F_1,F_2)$ denotes the group of homomorphisms
		from $F_1$ to $F_2$ preserving both the variety and the group structure. Since $F_1$ and $F_2$ are non-isogenous, $\Hom(F_1,F_2) = 0$, 
		which proves Claim~\ref{claim-isoprod}.
	\end{proof}

	Let $L$ be a line bundle on $X$.
	Claim~\ref{claim-isoprod} implies that
	$$L_{|\phi^{-1}(p)} \simeq L_{|F_1 \times \{u\}} \boxtimes L_{|\{v\} \times F_2},$$
	for any points $ u \in F_2$ and $v \in F_1$.
	
	For each $i = 1,2$, we choose a section $s_i : \PP^1 \to W_i$
	and let $\gs_i : W_i \to X$ be the induced section:
	$$\gs_1(w_1) \cnec (w_1,s_2(\phi_1(w_1))) \in  W_1 \times_{\PP^1} W_2,$$
	and similarly for $\gs_2$.
	We have
	\begin{align*}
		L_{|\phi^{-1}(p)} 
		& \simeq 
		L_{|F_1 \times \{s_1(p)\}} \boxtimes L_{|\{s_2(p)\} \times F_2} \\
		& \simeq (\gs_1^*L)_{|F_1} \boxtimes (\gs_2^*L)_{|F_2} \\
		& \simeq (p^*_1\gs_1^*L \otimes p^*_2\gs_2^*L)_{|\phi^{-1}(p)}.       
	\end{align*}
	Since $p \in \PP^1$ is very general, by applying \cite[Theorem 3.1 and Remark 3.3]{VoisinChDecomp} to the smooth part of the fibration $\phi: X \to \PP^1$, we obtain 
	$$L \sim_{\QQ}  p^*_1\gs_1^*L \otimes p^*_2\gs_2^*L\otimes \cO_X(D),$$
	for a divisor $D$ whose support is contained in
	a finite union of fibers of $\phi : X \to \PP^1$.
	Note that an irreducible component $R$ of a fiber of $\phi$ embeds in the product $\phi_1^{-1}(\phi(R))\times\phi_2^{-1}(\phi(R))$, of which at least one factor $\phi_i^{-1}(\phi(R))$ is smooth, hence irreducible. It follows that there is an irreducible component $R'$ of $\phi_j^{-1}(\phi(R))$ with $j = \{1, 2\} \setminus \{i\}$ such that $R=p_j^* R'$.
	Applying this to the irreducible components of $D$, we obtain that
	$$N^1(W_1)_{\RR} \times N^1(W_2)_{\RR} \xto{p_1^* + p_2^*} N^1(X)_{\RR}$$
	is surjective.
\end{proof}

\begin{lemma}\label{lem-nefdec2} For every $D \in \Nef(X)$, one can write $D = p_1^*D_1 + p_2^*D_2$, where $D_i \in \Nef(W_i)$.
\end{lemma}

\begin{proof} Lemma~\ref{lem-nefdec2} follows from Lemma~\ref{na1.1} and Theorem~\ref{thm-nefdec}. 
\end{proof}

\begin{theorem}[$=$ Theorem \ref{main}]\label{main2} Let $(X,\Delta)$ be a Schoen pair. Then 
	$$\Nef(X)=\Nefp(X)=\Nefe(X),$$
	and moreover, there exists a rational polyhedral fundamental domain for the action of $\Aut(X, \Delta)$ on $\Nefe(X)$. 
\end{theorem}

\begin{proof} 
	
	Since $\Nef(W_i)=\Nefp(W_i) = \Nefe(W_i)$ by Proposition \ref{prop-lef},
	we have, by Lemma \ref{lem-nefdec2}, $\Nef(X)=p_1^*\Nefp(W_1)+p_2^*\Nefp(W_2)\subset\Nefp(X)$, so $\Nef(X)=\Nefp(X)$. Similarly, we have $\Nef(X)=\Nefe(X)$. This proves the first assertion.
	
	Define the subgroups $H_i \le \Aut(W_i)$ by
	$$H_i =
	\begin{cases}
		\Aut(W_i/\PP^1), \,  \text{ if } W_i
		\text{ is a rational elliptic surface,} \\
		\{\mathrm{id}_{W_i}\}, \, \text{ otherwise. } 
	\end{cases}
	$$
	Then there exists a rational polyhedral cone
	$\Pi_i \subset \Nef^{+}(W_i)$ such that 
	$H_i \cdot \Pi_i$ contains $\Amp(W_i)$.
	Indeed, the case where 
	$W_i$ is a rational elliptic surface with $-K_{W_i}$ semiample follows from
	\cite[Theorem 8.2]{To08},
	and the other cases follow from
	Proposition~\ref{prop-lef}. 
	
	We claim that $H_1\times H_2\le\Aut(X,\Delta)$.
	Note that there exists a positive integer $m$ such that 
	$$\Delta = \frac{1}{m}\Delta_{m,X}$$ 
	for some $\Delta_{m,X}\in|-mK_X|$. 
	If neither $W_1$ nor $W_2$ is a
	rational elliptic surface, then $H_1 \times H_2$ is trivial by definition.
	If both $W_1$ and $W_2$ are rational elliptic surfaces, then $\Delta_{m,X}=0$ and clearly, $H_1\times H_2\le\Aut(X)$. Finally, if one of the $W_i$, say $W_1$, is a rational elliptic surface, and the other, say $W_2$, is not,
	then $\cO_X(-K_X) \simeq p_2^*\cO_{W_2}(-K_{Z_2} - D_2)$.
	Since $p_2$ is proper surjective with connected fibers, 
	the pullback $p_2^*$ induces an isomorphism
	$$H^0(X, p_2^*\OO_{W_2}(-m(K_{Z_2}+D_2))) \simeq 
	H^0(W_2, \OO_{W_2}(-m(K_{Z_2}+D_2))).$$ 
	So $\Delta_{m,X}=p_2^*\Delta_{m,W_2}$, 
	for some divisor $\Delta_{m,W_2}\in |-m(K_{Z_2}+D_2)|$. 
	Since $H_2 = \{ \mathrm{id}_{W_2}\}$ in this case, 
	it follows that $\Delta_{m,X}$ is invariant under $H_1 \times H_2$. This proves the claim. 
	
	It then follows from Corollary~\ref{cor-H1H2}
	that $\Nefe(X) = \Nefp(X)$ has 
	a rational polyhedral fundamental domain
	$\Pi$ for the $\Aut(X,\Delta)$-action. 
\end{proof}

\begin{remark}
	
	In \cite{GM93}, the authors verified the Cone Conjecture for a strict Calabi--Yau threefold $X = W_1 \times_{\PP^1} W_2$, where both $W_i$ are general rational elliptic surfaces with sections. They use the following identification shown by Namikawa~\cite[Proposition 2.2 and Corollary 2.3]{Na91}
	\[ \Aut(X)\cong \Aut(W_1)\times \Aut(W_2), \]
	which our proof bypasses, using Looijenga's result (Lemma \ref{looij}) instead.
\end{remark}

\begin{example}\label{ex-infinite} 
	Fix an integer $n\ge 3$.
	Let us explain how to choose $Z_1,Z_2,D_1,$ and $D_2$ so that our construction produces a strict Calabi--Yau manifold $X$ of dimension $n$, such that $\Nef(X)$ admits infinitely many extremal rays and $X$ satisfies the Cone Conjecture.
	We take $Z_1 = \PP^2$ and $D_1=\OO_{\PP^2}(3)$, so that $W_1$ is a general rational elliptic surface. We take $Z_2$ to be a Fano variety of dimension $n-1$ with $-K_{Z_2}$ globally generated (for example, $Z_2=\PP^{n-1}$), and we take $D_2= -K_{Z_2}$.
	
	The Schoen variety $X$ obtained from these choices is a strict Calabi--Yau manifold by Proposition \ref{su2.1bis}, and $\Nef(X)$ admits infinitely many extremal rays by Lemma~\ref{na1.1}, by the fact that $\Nef(W_1)$ admits infinitely extremal rays already, and by Corollary \ref{cor-extr}.  
\end{example}

We conclude with an unsurprising corollary of the fact that Schoen varieties satisfy the Cone Conjecture.

\begin{corollary}\label{cor_finite}
	Let $X$ be a Schoen variety. Then the group $\pi_0 \Aut(X)$ is finitely presented, and there are at most finitely many real forms for $X$, up to isomorphism.
\end{corollary}

\begin{proof} 
	The linear action $\rho: \Aut(X) \to \GL(N^1(X))$ induces and factorizes through an action $$\overline{\rho}: \pi_0\Aut(X) \to \GL(N^1(X)).$$ 
	We let $\Aut^{\ast}(X) = \rho (\Aut(X)) = \overline{\rho}(\pi_0\Aut(X))$. 
	
	Choose an effective $\QQ$-divisor $\gD$ on $X$ such that $(X,\gD)$ is a Schoen pair.
	By Theorem~\ref{main}, there exists a rational polyhedral cone $\Pi \subset \Nefp(X)$ such that
	$$\Amp(X)\subset \Aut(X,\Delta)\cdot \Pi\subset\Aut^{\ast}(X)\cdot \Pi.$$
	It follows from Proposition~\ref{pro-looij} that there is a rational polyhedral fundamental domain for the $\Aut^{\ast}(X)$-action on $\Nefp(X)$, and that the group $\Aut^{\ast}(X)$ is finitely presented. 
	By Fujiki--Liebermann's theorem~\cite[Corollary 2.11]{Br18}, the kernel $\Ker (\overline{\rho})$ is finite, and so the first claim follows from \cite[Corollary 10.2]{Jo97}. 
	
	The second claim follows from 
	Theorem~\ref{thm-realstr} below. 
\end{proof}

\begin{theorem}[{\cite[Theorem 1.6]{DGLOWY}}]\label{thm-realstr} Let $V$ be a smooth complex projective variety. Assume that $\Nefp(V)$ contains a rational polyhedral cone $\Pi$ such that
	$$\Amp(V)\subset \Aut(V)\cdot\Pi.$$ Then $V$ has at most finitely many mutually non-isomorphic real forms.
\end{theorem}

\bibliographystyle{abbrv}
\bibliography{main}

\end{document}